\documentclass[11pt,a4paper]{article}

\usepackage{cmap}
\usepackage[T1]{fontenc}

\usepackage{tocloft}

\usepackage{amsmath,amsthm,amssymb,latexsym,graphicx,mathrsfs}
\usepackage{secdot}
\sectiondot{subsection}
\usepackage{a4wide}
\usepackage{float}

\usepackage{mathtools}
\mathtoolsset{showonlyrefs}
\allowdisplaybreaks

\usepackage{marginnote}

\newcommand{\doi}[1]{\href{http://dx.doi.org/#1}{doi:#1}}

\usepackage{xcolor}
\usepackage{url}
\usepackage{hyperref}
\definecolor{ForestGreen}{rgb}{0.1,0.6,0.05}
\definecolor{EgyptBlue}{rgb}{0.063,0.1,0.6}
\definecolor{RipeOlive}{HTML}{556B2F}
\hypersetup{
	colorlinks=true,
	linkcolor=EgyptBlue,         
	citecolor=ForestGreen,
	urlcolor=RipeOlive
}

\usepackage[hyperpageref]{backref}
\usepackage{epstopdf}
\newcounter{dummy}
\usepackage{enumitem}
\makeatletter
\newcommand\myitem[1][]{\item[#1]\refstepcounter{dummy}\def\@currentlabel{#1}}
\makeatother

\usepackage{orcidlink}

\newtheorem{theorem}{Theorem}
\newtheorem{proposition}[theorem]{Proposition}
\newtheorem{lemma}[theorem]{Lemma}
\newtheorem{corollary}[theorem]{Corollary}

\theoremstyle{definition}

\newtheorem{openproblem}[theorem]{Open problem}

\numberwithin{equation}{section}
\numberwithin{theorem}{section}

\numberwithin{equation}{section}
\numberwithin{theorem}{section}

\newenvironment{proof*}[1]{\begin{trivlist}\item[\hskip%
		\labelsep{{\bf Proof of \/{\rm\bf #1.}}~}]\rm}%
	{\hfill\qed\rm\end{trivlist}}

\newcommand{\W}{W_0^{1,p}}

\newcommand{\intO}{\int_\Omega}

\newcommand{\Om}{\Omega}

\title{
	\vspace*{-2cm}
	Remarks on the antimaximum principle
} 
\author{Vladimir Bobkov\\}
\date{}

\AtEndDocument{%
	\par
	\bigskip
	\bigskip
	\begin{tabular}{@{}l@{}}%
		(V.~Bobkov)\\[0.2em]
		\textsc{Institute of Mathematics, Ufa Federal Research Centre, RAS}\\
		\textsc{Chernyshevsky str. 112, Ufa 450008, Russia;}\\
		\textsc{Ufa University of Science and Technology}\\
		\textsc{Zaki Validi str. 32, Ufa 450076, Russia.}\\[0.3em]
		\orcidlink{0000-0002-4425-0218} 0000-0002-4425-0218\\[0.3em]
		\textit{E-mail address}: \texttt{bobkov@matem.anrb.ru}, \texttt{bobkovve@gmail.com}
\end{tabular}}

\begin{document}
	\maketitle
	\vspace*{-5ex}
	\begin{abstract}
		
		We present several observations on the antimaximum principle (AMP) for the model problem $-\Delta_p u = \lambda |u|^{p-2} u + f$ in a bounded smooth domain $\Omega$, subject to the zero Dirichlet boundary conditions, and where the source function $f$ is nontrivial, nonnegative, and sufficiently regular. 
		Denote by $\lambda_f$ the endpoint of validity of the AMP, so that every solution of the problem is negative in $\Omega$ for any $\lambda \in (\lambda_1,\lambda_f)$. 
		Our discussion covers the following aspects: identification of a class of sources over which the AMP is uniform, lower semicontinuity of the map $f \mapsto \lambda_f$, bounds on $\lambda_f$, the nonexistence of negative solutions for sufficiently large $\lambda$ (extended AMP), the anticomparison principle, and the weakening of the source regularity from the Lebesgue to Morrey spaces. 
		Some of the results are stated only in the linear case $p=2$. 
		As a part of the discussion, we provide a few related open problems.

		
	\end{abstract}
	
	\begin{quote}	
		\tableofcontents	
		\addtocontents{toc}{\vspace*{-2ex}}
	\end{quote}

\section{Introduction and main results}\label{sec:introduction}

	Consider the boundary value problem
	\begin{equation}\label{D}
	\tag{$\mathcal{D}_\lambda$}
		\left\{
		\begin{aligned}
			-\Delta_p u &= \lambda |u|^{p-2}u + f(x) 
			&&\text{in}\ \Omega, \\
			u&=0 &&\text{on}\ \partial \Omega,
		\end{aligned}
		\right.
	\end{equation}
	where $p>1$ and $\Omega \subset \mathbb{R}^N$ is a bounded domain of class $C^{1,1}$, $N \geq 1$.
	If not explicitly mentioned otherwise, we always assume that 
	\begin{equation}\label{eq:F}
	f \geq 0, ~ f\not\equiv0 ~~\text{in}~ \Omega, 
	\quad \text{and} \quad 
	f \in L^\gamma(\Omega) ~~\text{for some}~ \gamma>N. 
	\end{equation}
	Hereinafter, $\gamma$ is fixed. 
	Denote by $\lambda_1$ and $\lambda_2$ the \textit{first} and \textit{second} eigenvalues of the $p$-Laplacian in $\Omega$. 
	Sometimes we will also use the expanded notation $\lambda_1(\Omega)$ to represent the dependence on $\Omega$. 
	These eigenvalues admit the following characterizations:
	\begin{equation}\label{eq:lambda1}
	\lambda_1 = 
	\inf\left\{
	\frac{\int_\Omega |\nabla u|^p \, dx}{\int_\Omega |u|^p \, dx}:~
	u \in W_0^{1,p}(\Omega) \setminus \{0\}
	\right\}
	\end{equation}
	and (see, e.g., \cite{Bob1})
	\begin{equation}\label{eq:lambda2-char}
			\lambda_2 
			=
			\inf 
			\left\{
			\max
			\left\{
			\frac{\int_\Omega |\nabla u^+|^p \,dx}{\int_\Omega |u^+|^p \, dx},
			\frac{\int_\Omega |\nabla u^-|^p \,dx}{\int_\Omega |u^-|^p \, dx}
			\right\}:~
			u \in W_0^{1,p}(\Omega), ~ u^\pm \not\equiv 0
			\right\},
		\end{equation}
	where we use the notation $a^\pm = \max\{\pm a, 0\}$. 
	It is known that $0 < \lambda_1 < \lambda_2$ \cite{anane1987}, the minimizer $\varphi_1 \in W_0^{1,p}(\Omega)$ of $\lambda_1$ (a.k.a.\ the first eigenfunction) is a weak solution of $(\mathcal{D}_{\lambda_1})$ with $f=0$, it is unique modulo scaling \cite{Alleg,Lind}, $\varphi_1 \in C^{1,\beta}(\overline{\Omega})$ for some $\beta \in (0,1)$ \cite{Lieberman}, and it can be chosen such that
	\begin{equation}\label{eq:bpl}
	\varphi_1 > 0 \quad \text{in}~ \Omega \quad \text{and} \quad \partial_\nu \varphi_1 < 0 \quad \text{on}~ \partial \Omega,
	\end{equation}
	where $\nu$ is the outer unit normal vector to $\partial \Omega$ \cite{vaz}.
	We always assume that $\varphi_1$ is normalized in $L^p(\Omega)$ as  $\|\varphi_1\|_{p}=1$.  
	Hereinafter, except in Section~\ref{sec:Morrey}, we always work with \textit{weak} solutions of \eqref{D}. 
	Such solutions do exist for any $\lambda \in (\lambda_1,\lambda_2)$, see, e.g., \cite[Lemma~4.6]{BDI} or \cite[Theorem~2.20]{BobkovTanakaAM}.
	Moreover, any weak solution belongs to $C^{1,\beta}(\overline{\Omega})$ for some $\beta \in (0,1)$, see \cite[Propositions~A.1 and~A.3]{BobkovTanakaAM}.
		
	\medskip
	The sign of a solution is among fundamental qualitative properties of \eqref{D}.   
	It is well known that the following \textit{maximum principle} \eqref{MP} is valid, see, e.g., \cite{vaz}: 
	\begin{equation}\label{MP}
		\tag{$\mathcal{MP}$}
		\text{If $\lambda < \lambda_1$, then any solution $u$ of \eqref{D} satisfies $u>0$ in $\Omega$ and $\partial_\nu u <0$ on $\partial \Omega$}.
	\end{equation}
	In contrast, it was observed by Cl\'ement \& Peletier \cite{CP} for $p=2$ and by Fleckinger et al.\ \cite{FGTT} for $p>1$ (see also \cite{BobkovTanakaAM} for the regularity assumption as in \eqref{eq:F}) that the following \textit{antimaximum principle} \eqref{AMP} holds:  
	\begin{equation}
	\tag{$\mathcal{AMP}$}
	\label{AMP}
	\begin{split}
		&\text{There exists $\lambda_f > \lambda_1$ such that if $\lambda \in (\lambda_1, \lambda_f)$, then}\\
		&\text{any solution $u$ of \eqref{D} satisfies $u < 0$ in $\Omega$ and $\partial_\nu u > 0$ on $\partial \Omega$.}
		\end{split}
	\end{equation}
This version of \ref{AMP} is sometimes called \textit{strong}. 
Weaker versions are also of independent interest; in these versions, the inequality $\partial_\nu u > 0$ on $\partial \Omega$ is omitted and/or $u<0$ is replaced by $u \leq 0$ in $\Omega$. 

We will always assume that $\lambda_f$ is the maximal value such that \ref{AMP} holds. 
In the linear case $p=2$, the Fredholm alternative states that \eqref{D} has a solution for any nonresonant $\lambda$, which yields $\lambda_f \leq \lambda_2$, see, e.g., \cite{BDI}.  
However, the nonlinear case $p \neq 2$ raises the following subtlety: we do not know whether, in general, \eqref{D} has solutions for \textit{at least some} $\lambda \geq \lambda_2$. 
It might hypothetically happen that \ref{AMP} as written above holds for any $\lambda>\lambda_1$, while there are no solutions of \eqref{D} when $\lambda$ is sufficiently large. 
We prefer to avoid such a ``trivial'' scenario, and hereinafter we \textit{always} require in \ref{AMP} the existence of solutions of $(\mathcal{D}_{\widetilde\lambda_n})$ for a sequence $\widetilde\lambda_n \nearrow \lambda_f$. 
In that way, $\lambda_f$ can be equivalently defined as 
\begin{align}
\lambda_f = 
\sup
\Big\{
a>\lambda_1:~ 
&\text{$(\mathcal{D}_{a})$ has a solution, and for any $\lambda \in (\lambda_1,a]$}\\
&\text{every solution $u$ of \eqref{D} satisfies $u < 0$ in $\Omega$ and $\partial_\nu u > 0$ on $\partial \Omega$}
\Big\}.~~~~
\label{eq:lambdaf}
\end{align}
This definition refines the one considered in \cite{BDI} by the additional existence assumption. 

The antimaximum principle was investigated in \cite{CP} in the linear case $p=2$ via the Green's function analysis. 
The general nonlinear case $p>1$ was covered in \cite{FGTT} by different arguments based on  compactness analysis, and the core of that approach is adopted in the present work. 
After \cite{CP,FGTT}, various aspects of \ref{AMP} were subsequently developed, and we refer to \cite{BDI,BobkovTanakaAM,DFG,FleckGossThel,FHT2014,hamel-nadir,sweers} for some of them. 
Extensions to more general domains and boundary conditions were studied in \cite{ADSS,ACG,BvonBR}. 
Alternative approaches to \ref{AMP} were considered in, e.g., \cite{AroraGluck,korman2019,shi,takac-abstract}. 
Particular attention was paid to \ref{AMP} in indefinite weight settings, see \cite{GGP,IR,pinch,TanakaAMP2012}. 
We also refer to the overview \cite{M}. 

Although it may seem that \ref{AMP} has been comprehensively studied, several of its fundamental properties have not been investigated, especially in the nonlinear case.  
For instance, analytic properties of the map $f \mapsto \lambda_f$ do not seem to have been systematically studied so far, and even the finiteness of $\lambda_f$ remains unknown in the general case $p \neq 2$, see \cite{BDI} and Section~\ref{sec:properties} below for more details. 
This is mainly due to the nonconstructive definition of $\lambda_f$; compare \eqref{eq:lambdaf} with \eqref{eq:lambda1} and \eqref{eq:lambda2-char}. 

The aim of the present work is to obtain some information in this regard. 
Apart from investigating the analytic properties of $\lambda_f$, such as uniform lower bounds on $\lambda_f$ on classes of source functions (a.k.a.\ uniform \ref{AMP}), lower semicontinuity, monotonicity, and upper bounds, 
we also generalize the notion of \ref{AMP} by introducing the \textit{extended} \ref{AMP} (which is about the nonexistence of negative solutions for sufficiently large $\lambda$) and the \textit{anticomparison} principle. 
Moreover, we show that the Lebesgue regularity of $f$ in the assumption \eqref{eq:F} can be weakened to an appropriate regularity in the Morrey spaces, so that \ref{AMP} holds for strong solutions of \eqref{D}. 

The work has the following structure. 
In Sections~\ref{sec:properties}, \ref{sec:EAMP}, \ref{sec:anticomparison}, \ref{sec:Morrey}, we provide our main results on \ref{AMP}.
Section~\ref{sec:auxiliary} contains a few important auxiliary statements.  
Finally, Section~\ref{sec:proofs} is devoted to proofs of the main results.

\subsection{Properties of \texorpdfstring{$\lambda_f$}{lambda-f}}\label{sec:properties}

It is known that $\lambda_f$ depends on $f$ in such a way that it cannot be separated from $\lambda_1$ \textit{uniformly} with respect to all sources satisfying \eqref{eq:F}, see \cite{ACG,BDI,DFG}. 
In contrast, such a uniform version of \ref{AMP} is known under the Neumann boundary conditions. 
Namely, let us denote by $\mu_f$ the endpoint of validity of \ref{AMP} in the Neumann case, in analogy with $\lambda_f$ in \eqref{eq:lambdaf}. 
It was shown in \cite{CP} that $\mu_f \geq \mu_2/4$ when $p=2$ and $N=1$, where $\mu_2$ is the second (or, equivalently, the first nonzero) eigenvalue of the Neumann Laplacian.
Later, it was proved in \cite{ACG} that if $p>N$, then 
\begin{equation}\label{eq:muf-bound}
\mu_f > 
\inf\left\{
\frac{\int_\Omega |\nabla u|^p \, dx}{\int_\Omega |u|^p \, dx}:~
u \in W^{1,p}(\Omega) \setminus \{0\}
\text{ and }
u
\text{ vanishes on some ball in }
\Omega
\right\} > 0,
\end{equation}
and no uniform lower bound for $\mu_f$ over $f$ is possible when $p \leq N$.

In our first result, we identify a natural class of sources in which 
\ref{AMP} (in the present Dirichlet setting) is uniform. 
For this purpose, we note that the map $f \mapsto \lambda_f$ is zero-homogeneous (i.e., scale-invariant), that is, 
\begin{equation}\label{eq:lambdacf}
\lambda_{cf}=\lambda_f \quad \text{for any}~ c > 0,
\end{equation}
and define the zero-homogeneous functional 
$$
\Theta(f)
=
\frac{\int_\Omega f\varphi_1\,dx}{\|f\|_\gamma}.
$$
Since $\varphi_1>0$ in $\Omega$, we have $\Theta(f) > 0$ for every $f$ satisfying \eqref{eq:F}.
For any $\eta>0$, consider the class of sources
$$
\mathcal{U}_\eta
=
\{f:~ f ~\text{satisfies \eqref{eq:F} and}~ \Theta(f) \geq \eta \}. 
$$
In Lemma~\ref{lem:properties-of-U} below, we show that $\mathcal{U}_\eta \cup \{0\}$ is a convex cone having a weak closure property. 

The following theorem states that \ref{AMP} is uniform over $\mathcal{U}_\eta$. 
\begin{theorem}\label{thm:UAMP}
For any $\eta>0$, there exists $\varepsilon>0$ such that 
$$
\lambda_f \geq \lambda_1 + \varepsilon
\quad \text{for every}~ f \in \mathcal{U}_\eta.
$$
\end{theorem}
In other terms, Theorem~\ref{thm:UAMP} states that if $\{f_n\}$ satisfies \eqref{eq:F}, then 
\begin{equation}\label{eq:uniformAMPoneside}
\lambda_{f_n} \to \lambda_1 \quad \text{implies} \quad \Theta(f_n) \to 0.
\end{equation}
In Lemma~\ref{lem:implication} below, we show that the reverse implication in \eqref{eq:uniformAMPoneside} is not generally possible, cf.\ \eqref{eq:uniformAMPanotherside}. 
This means that the class $\mathcal{U}_\eta$ is sufficient, but not necessary for the validity of the uniform \ref{AMP}.

The result of Theorem~\ref{thm:UAMP} helps to obtain a kind of weak lower semicontinuity of the map $f \mapsto \lambda_f$.
\begin{theorem}\label{thm:LSC}
Let $\{f_n\}$ and $f$ satisfy \eqref{eq:F}, and $f_n\to f$ weakly in $L^\gamma(\Omega)$.
Then
$$
 \min\{\lambda_f,\lambda_2\} \leq \liminf_{n\to+\infty}\lambda_{f_n}.
$$
Consequently, if $\lambda_f\le\lambda_2$, then the map $g\mapsto\lambda_g$ is weakly lower semicontinuous at $f$.
\end{theorem}

The inequality $\lambda_f\le\lambda_2$ is true, for instance, in the linear case $p=2$, see, e.g., \cite{BDI}. 
However, it is not even known whether $\lambda_f < +\infty$ in the general nonlinear setting. 
Let us discuss the issue of boundedness of $\lambda_f$ in more detail. 
Consider the critical value
\begin{equation}\label{eq:l*}
\lambda^*_f := \inf\left\{
\frac{\int_\Omega |\nabla u|^p \, dx}{\int_\Omega |u|^p \, dx}:~
\int_\Omega f u \, dx = 0, 
~
u \in W_0^{1,p}(\Omega) \setminus \{0\}
\right\}.
\end{equation}
It is proved in \cite[Theorem~1.1]{BDI} that $\lambda_f^* \in (\lambda_1, \lambda_2]$ and the following assertions hold:
	\begin{enumerate}[label={\rm(\roman*)}]
		\item\label{thm:1:2} If $\lambda_f^* < \lambda_2$, then $\lambda_f < \lambda_f^*$. 
		In particular, if $\{f_n\}$ satisfies \eqref{eq:F}, then (cf.\ \eqref{eq:uniformAMPoneside})
		\begin{equation}\label{eq:uniformAMPanotherside}
		\lambda^*_{f_n} \to \lambda_1 \quad \text{implies} \quad \lambda_{f_n} \to \lambda_1.
		\end{equation}
		\item\label{thm:1:1} If $p = 2$ or $N=1$, then $\lambda_f \leq \lambda_f^*$.
	\end{enumerate}
Notice that $\lambda_f < \lambda_f^* < \lambda_2$, provided there exists a second eigenfunction $\varphi_2$ such that $\int_\Omega f \varphi_2 \,dx \neq 0$, see \cite[Proposition~1.2]{BDI}.

We add one more case guaranteeing the boundedness of $\lambda_f$.  
Denote
$$
\Omega_f^0 = \bigcup 
\left\{
U \subset \Omega:~ U ~\text{is open and}~ f=0 ~\text{a.e.\ in}~ U
\right\}.
$$
If $f$ is continuous, then we equivalently have $\Omega_f^0 = \text{Int}(\{x \in \Omega:\, f(x)=0\})$.
\begin{proposition}\label{prop:upperbound}
Let $\Omega_f^0 \neq \emptyset$. 
Then $\lambda_f \leq \lambda_1(\Omega_f^0) < +\infty$.  
\end{proposition}

\begin{openproblem}
In view of \cite{BDI}, Proposition~\ref{prop:upperbound}, and the definition \eqref{eq:lambdaf}, it remains unknown whether $\lambda_f < +\infty$ if the following assumptions are simultaneously satisfied: $p \neq 2$, $N\geq 2$, $\lambda_f^* = \lambda_2$, $\Omega_f^0 = \emptyset$, and  $(\mathcal{D}_{\widetilde\lambda_n})$ has a solution for a sequence $\widetilde\lambda_n \nearrow \lambda_f$. 
\end{openproblem}

As for the lower bounds on $\lambda_f$, we can refer only to the following inequality obtained in \cite{FHT2014} in the linear case $p=2$:
\begin{equation}\label{eq:lambdaflowerbound}
\lambda_f 
\geq 
\min \left\{
\Lambda, \lambda_1 + \frac{K \alpha}{\|f^\perp\|_\gamma}
\right\}
\quad \text{for any fixed}~  \Lambda \in (\lambda_1,\lambda_2),
\end{equation}
where $\alpha>0$ and $f^\perp \in L^\gamma(\Omega)$ are defined through the $L^2(\Omega)$-orthogonal decomposition $f = \alpha \varphi_1 + f^\perp$, and the constant $K=K(\Omega,\Lambda,\gamma)$ does not depend on $f$, yet $K$ is not explicitly quantified. 

In the linear case $p=2$, we provide the following result (cf.\ Theorem~\ref{thm:UAMP}), which might be used to obtain more explicit bounds on $\lambda_f$ in certain regimes. 
\begin{proposition}\label{prop:convexity}
Let $p=2$. 
For any $a\in(\lambda_1,\lambda_2)$, define
$$
\mathcal A_a
=
\left\{
f:~ f ~\text{satisfies \eqref{eq:F} and }~
a < \lambda_f
\right\}.
$$
Then $\mathcal A_a \cup \{0\}$ is a convex cone: if $f,g\in\mathcal A_a$ and $s,t\ge0$ with $s+t>0$, then $sf+tg\in\mathcal A_a$. 
Consequently, for any $f,g$ satisfying \eqref{eq:F}, 
$$
\min\{\lambda_f,\lambda_g\}
\leq
\lambda_{sf+tg},
$$
that is, the map $f \mapsto \lambda_f$ is quasiconcave.
In particular, if $f=\varphi_1$, then $\lambda_{f}=\lambda_2$, and hence
$$
\lambda_{g} \leq \lambda_{s\varphi_1 + t g}.
$$
\end{proposition}

\begin{openproblem}
Find an (explicitly) quantified lower bound on $\lambda_f$ for $p \neq 2$. 
\end{openproblem}

Finally, we note that the map $f \mapsto \lambda_f$ is not monotone in any reasonable sense, at least for $p=2$.  
For instance, let us take some $\varepsilon \in (0,1)$, some second eigenfunction $\varphi_2$ normalized as $\|\varphi_2\|_\infty=1$, and put $h=\varphi_1(1+ \varepsilon \varphi_2)$. 
Then we have 
\begin{equation}\label{eq:phifphi}
0<(1-\varepsilon) \varphi_1 \leq h \leq (1+\varepsilon) \varphi_1 \quad \text{in}~ \Omega. 
\end{equation}
Noting that $\int_\Omega h \varphi_2 \,dx = \varepsilon \intO \varphi_1 \varphi_2^2 \,dx > 0$, we deduce from \cite[Proposition~1.2]{BDI} above that $\lambda_h < \lambda_h^* < \lambda_2$. 
However, we also have $\lambda_{\varphi_1} = \lambda_2$ in the linear case $p=2$, see \cite[Lemma~2.1 (v)]{BDI}. 
Thus, we deduce from these facts and \eqref{eq:lambdacf}, \eqref{eq:phifphi} that $f \leq g$ implies either $\lambda_f \leq \lambda_g$ or $\lambda_f \geq \lambda_g$, depending on the choice of $f,g$.

\subsection{Extended AMP}\label{sec:EAMP}

The proof of Proposition~\ref{prop:upperbound} gives more than just an upper bound for $\lambda_f$. 
To make this precise, we introduce the following \textit{extended} antimaximum principle: 
\begin{equation}
	\label{EAMP}
	\begin{split}
		&\text{There exists $\widetilde{\lambda}_f > \lambda_1$ such that if $\lambda > \widetilde{\lambda}_f$, then} \\
		&\text{no solution of \eqref{D} is negative in $\Omega$.}
		\end{split}
	\end{equation}
We always assume that $\widetilde{\lambda}_f$ is a minimal value such that the extended $\mathcal{AMP}$ \eqref{EAMP} holds. 
Equivalently, $\widetilde{\lambda}_f$ can be defined as 
\begin{equation}\label{eq:lambdaf-eamp}
\widetilde{\lambda}_f 
=
\sup
\{\lambda>\lambda_1:~ \eqref{D} ~\text{has a negative solution}\}.
\end{equation}
By comparing \eqref{eq:lambdaf} and \eqref{eq:lambdaf-eamp}, we get $\lambda_f \leq \widetilde{\lambda}_f$. (Here, we used the existence assumption in \eqref{eq:lambdaf}.)
The arguments of Proposition~\ref{prop:upperbound} actually provide a bound on $\widetilde{\lambda}_f$.
\begin{corollary}\label{cor:eamp}
Let $\Omega_f^0 \neq \emptyset$. 
Then $\widetilde{\lambda}_f \leq \lambda_1(\Omega_f^0) < +\infty$.  
\end{corollary}

Notice that the proof of Corollary~\ref{cor:eamp} cannot be directly applied if the \textit{negativity} in \eqref{EAMP} (or \eqref{eq:lambdaf-eamp}) is replaced by the \textit{nonpositivity}. 
Such a modified version of the extended $\mathcal{AMP}$ is of independent interest. 

Despite Corollary~\ref{cor:eamp}, it might occur that $\widetilde{\lambda}_f = +\infty$. Indeed, let $f = \varphi_1^{p-1}$. 
Then, for any $\lambda>\lambda_1$, a negative solution of \eqref{D} is given by
$$
u = -\frac{\varphi_1}{(\lambda-\lambda_1)^{1/(p-1)}} \quad \text{in}~\Omega,
$$
implying $\widetilde{\lambda}_f = +\infty$. 
In the linear case $p=2$, we describe a larger sufficient class, cf.\ \cite[Theorem~4.1]{FleckGossThel}.

\begin{proposition}\label{prop:eamp-linear}
Let $p=2$. 
Let $\{\varphi_k\}$ be a basis of eigenfunctions. 
Let $f=\sum_{i=1}^m c_i\varphi_i$ for some $m\in\mathbb{N}$, where $c_1,\ldots,c_m\in\mathbb{R}$ are chosen so that $f\geq \kappa\varphi_1$ in $\Omega$ for some $\kappa>0$. 
Then $\widetilde\lambda_f=+\infty$. 
\end{proposition}

We refer to \cite{FleckGossThel} for some other extensions of \ref{AMP} beyond $\lambda_1$ in the linear case $p=2$. 

\begin{openproblem}
Characterize an optimal class of sources guaranteeing $\widetilde\lambda_f=+\infty$ in the linear case $p=2$. 
Describe a nontrivial class of sources guaranteeing $\widetilde\lambda_f=+\infty$ in the nonlinear case $p \neq 2$. 
\end{openproblem}

\subsection{Anticomparison principle}\label{sec:anticomparison}

Consider now two problems of the type \eqref{D}:
\begin{gather}
\label{eq:D1}
-\Delta_p u = \lambda |u|^{p-2}u + f(x) ~\text{in}~ \Omega, \quad u=0 ~\text{on}~\partial\Omega,\\ 
\label{eq:D2}
-\Delta_p v = \lambda |v|^{p-2}v + g(x) ~\text{in}~ \Omega, \quad v=0 ~\text{on}~\partial\Omega,
\end{gather}
where $f,g$ satisfy \eqref{eq:F}.
Let $u,v \in W_0^{1,p}(\Omega)$ be any solutions of \eqref{eq:D1}, \eqref{eq:D2}, respectively.  
In the linear case $p=2$, the maximum principle \ref{MP} is equivalent to the \textit{comparison} principle. 
That is to say, 
\begin{equation}\label{eq:comparison}
\begin{split}
&\text{If $\lambda < \lambda_1$ and $f \leq g$, $f \not\equiv g$ in $\Omega$, then}\\
& 0<u<v ~\text{in}~\Omega
~\text{and}~ \partial_\nu v< \partial_\nu u < 0 ~ \text{on}~ \partial \Omega.
\end{split}
\end{equation}
In the nonlinear case $p \neq 2$, the situation is considerably more complicated and, in general, the comparison principle might be violated. 
Nonetheless, it was proved in \cite{CuestaTakac-strongcomparison1} that \eqref{eq:comparison} remains valid under the additional assumptions $\lambda \geq 0$, $f,g \in L^\infty(\Omega)$, and $\Omega$ is of class $C^{2,\alpha}$, $\alpha \in (0,1)$. 

We would like to investigate what happens with the comparison principle beyond $\lambda_1$, by analogy with the relation between \ref{MP} and \ref{AMP}. 
In this way, we state the following \textit{anticomparison} principle. 

\begin{theorem}\label{thm:anticomparison}
Let $f,g$ satisfy \eqref{eq:F}. 
Let $f \leq g$ and $f \not\equiv g$ in $\Omega$.
Then there exists $\lambda^* \in (\lambda_1, \min\{\lambda_f, \lambda_g\})$ such that, for every $\lambda \in (\lambda_1,\lambda^*)$, any solutions $u,v \in W_0^{1,p}(\Omega)$ of \eqref{eq:D1}, \eqref{eq:D2} satisfy
\begin{equation}\label{eq:thm:ACP}
 v<u<0 ~\text{in}~\Omega
\quad \text{and} \quad \partial_\nu v > \partial_\nu u > 0 ~ \text{on}~ \partial \Omega.
\end{equation}
\end{theorem}

\begin{proposition}\label{prop:anticomparison-linear}
Let the assumptions of Theorem~\ref{thm:anticomparison} be satisfied. 
If $p=2$, then the anticomparison principle \eqref{eq:thm:ACP} holds for any $\lambda \in (\lambda_1, \min\{\lambda_{f},\lambda_{g-f}\})$, and $\min\{\lambda_{f},\lambda_{g-f}\}$ is the endpoint of its validity.  
\end{proposition}

In the linear case $p=2$, the classical Fredholm alternative guarantees the existence of a unique solution of \eqref{D} when $\lambda$ is not an eigenvalue. 
Assuming that $f$ satisfies \eqref{eq:F} and $f \in L^\infty(\Omega)$, the uniqueness remains valid for any  $\lambda<\lambda_1$ and $p > 1$, see \cite{CuestaTakac-strongcomparison1}. 
In contrast, when $p \neq 2$ and $f$ is sign-changing, \eqref{D} might possess several distinct solutions when $\lambda<\lambda_1$, see \cite[Remark~1]{CuestaTakac-strongcomparison1}. 
We refer to the surveys \cite{drabek,takac-lec1} and extensive bibliographies therein for a number of involved results on the nonlinear Fredholm alternative. 

Although Theorem~\ref{thm:anticomparison} applies to \textit{any} solutions of \eqref{eq:D1}, \eqref{eq:D2}, the following question appears to be natural. 

\begin{openproblem}
Let $p \neq 2$. Prove (or disprove) the uniqueness of a solution of \eqref{D} as $\lambda \searrow \lambda_1$, under the assumption \eqref{eq:F}. 
\end{openproblem}

\subsection{AMP in Morrey spaces}\label{sec:Morrey}

In this section, we deal only with the linear case $p=2$, for simplicity. 
What matters most for the validity of \ref{AMP} is the uniform $C^{1,\beta}(\overline{\Omega})$-regularity of solutions of \eqref{D}. 
Such regularity takes place if $f \in L^\gamma(\Omega)$, $\gamma > N$, while 
\ref{AMP} does not generally hold for $f \in L^N(\Omega)$, see an explicit example in \cite{sweers}. 
At the same time, the scale of the Lebesgue spaces is not the only sufficient class for obtaining the required regularity. 
Another sufficient scale is described in terms of the Morrey spaces. 

Let $q \in [1,+\infty)$ and $\mu \in [0,N]$. 
We define the Morrey space $L^{q,\mu}(\Omega)$ as the set of functions $f \in L^1_{\text{loc}}(\Omega)$ such that
$$
 \|f\|_{q,\mu}^q
 :=\sup_{x\in\Omega,\,0<r\le \operatorname{diam}\Omega}
 r^{-\mu}\int_{\Omega\cap B_r(x)}|f|^q\,dy < +\infty. 
$$
We refer to \cite{PKJF,SFH} for a systematic study of these spaces. 
In particular, let us mention that $L^{q,\mu}(\Omega)$ is a Banach space with the norm $\|\cdot\|_{q,\mu}$, and
\begin{equation}\label{eq:propertiesMorrey}
L^{q,\mu}(\Omega) \subset L^{q}(\Omega),
\quad
L^{q,0}(\Omega)=L^q(\Omega),
\quad
L^{q,N}=L^\infty(\Omega).
\end{equation}

We now state \ref{AMP} for strong solutions of \eqref{D} when the source function belongs to the Morrey space. 
Hereinafter, for a given $q>1$, by the strong solution of \eqref{D} we understand a function $u \in W^{2,q}(\Omega)\cap W^{1,q}_0(\Omega)$ which satisfies the equation in \eqref{D} a.e.\ in $\Omega$. 
\begin{theorem}\label{thm:AMP-morrey}
Let $p=2$. 
Assume that $\int_\Omega f \varphi_1 \,dx > 0$
and $f\in L^{q,\mu}(\Omega)$, 
where
\begin{equation}\label{eq:morrey:assumptions}
1<q \leq N \quad \text{and} \quad N-q< \mu < N.
\end{equation}
Then \ref{AMP} holds for strong solutions  
$u \in W^{2,q}(\Omega)\cap W^{1,q}_0(\Omega)$ of \eqref{D}. 
Such solutions exist and are unique for any $\lambda \in (\lambda_1,\lambda_2)$.  
\end{theorem}

It is not hard to deduce from \cite[Section~1.2.3, Example~8 and Theorem~18]{SFH} that  
there exists a nonnegative function $f\in L^{q,\mu}(\Omega)$ for some $q,\mu$ satisfying \eqref{eq:morrey:assumptions}, but $f \notin L^N(\Omega)$. 
That is, the regularity assumption of Theorem~\ref{thm:AMP-morrey} is weaker than the usual Lebesgue regularity in \eqref{eq:F}. 

We intentionally made this subsection the last one, in order to avoid overcomplications with translating the known results on \ref{AMP} and those established above to the weaker regularity assumption of Theorem~\ref{thm:AMP-morrey}. 
We leave the corresponding developments for interested readers.

\section{Auxiliary results}\label{sec:auxiliary}

In this section, we provide a few auxiliary statements. 
The proofs of Proposition~\ref{prop:AMP-convergence} and Theorem~\ref{thm:AMP} are mainly inspired by the arguments from \cite{BobkovTanakaAM,FGTT}. 
Theorem~\ref{thm:AMP} identifies the exact asymptotic profile of solutions of \eqref{D} as $\lambda \to \lambda_1$. 
It is of independent interest and can be seen as a generalization and refinement of the classical \ref{AMP}. 

In what follows, we will frequently consider some sequences $\{f_n\} \subset L^\gamma(\Omega)$, 
$\{\widetilde{\lambda}_n\} \subset (\lambda_1,\lambda_2)$, and $\{u_n\}\subset W_0^{1,p}(\Omega)$, where $u_n$ denotes a weak solution of the problem
\begin{equation}\label{eq:Dn}
 -\Delta_p u_n=\widetilde{\lambda}_n |u_n|^{p-2}u_n+f_n(x) \quad \text{in}~\Omega,  
 \quad u_n=0 \quad \text{on }\partial\Omega.
\end{equation}

\begin{proposition}\label{prop:AMP-convergence}
Let $\{f_n\}$ and $f$ satisfy \eqref{eq:F}, and $f_n\to f$ weakly in $L^\gamma(\Omega)$. 
Let $\{\widetilde{\lambda}_n\} \subset (\lambda_1,\lambda_2)$ and $\lambda^* \in [\lambda_1,\lambda_2)$ be such that $\widetilde{\lambda}_n \to \lambda^*$. 
Let $u_n$ be a solution of \eqref{eq:Dn}. 
Then $\{u_n\}$ is bounded in $L^p(\Omega)$ if and only if $\lambda^* > \lambda_1$.
Moreover, if $\{u_n\}$ is bounded in $L^p(\Omega)$, then 
$u_n \to u$ in $C^{1}(\overline{\Omega})$, 
up to a subsequence, where $u$ is a weak solution of
\begin{equation}\label{eq:D0}
 -\Delta_p u= \lambda^* |u|^{p-2}u+f(x) \quad \text{in}~\Omega,  
 \quad u=0 \quad \text{on }\partial\Omega.
\end{equation}
\end{proposition}
\begin{proof}
Notice that since $\{f_n\}$ weakly converges in $L^\gamma(\Omega)$, it is bounded therein. 
Since $u_n$ is a solution of \eqref{eq:Dn}, it satisfies the energy identity
\begin{equation}\label{eq:energy1}
\int_\Omega |\nabla u_n|^p \,dx
=
\widetilde{\lambda}_n \int_\Omega |u_n|^p \,dx + \int_\Omega f_n u_n \,dx.
\end{equation}
Using the H\"older inequality, the convergence of $\{\widetilde{\lambda}_n\}$, the boundedness of $\{f_n\}$ in $L^\gamma(\Omega)$, and the continuity of the embedding $W^{1,p}_0(\Omega)\hookrightarrow L^{\gamma'}(\Omega)$, where $\gamma'=\gamma/(\gamma-1)$, we obtain
\begin{equation}\label{eq:proof:amp0}
\|\nabla u_n\|_p^p
\leq
\widetilde{\lambda}_n
\|u_n\|_p^p
+
\|f_n\|_\gamma \|u_n\|_{\gamma'}
\leq 
C\|u_n\|_p^p
+
C\|\nabla u_n\|_p,
\end{equation}
where $C>0$ does not depend on $n$.
 
If $\{u_n\}$ is bounded in $L^p(\Omega)$, then we deduce from \eqref{eq:proof:amp0} that $\{u_n\}$ is also bounded in $W^{1,p}_0(\Omega)$. 
Therefore, there exists $u \in W^{1,p}_0(\Omega)$ such that $u_n \to u$ weakly in $W^{1,p}_0(\Omega)$ and strongly in $L^{\gamma'}(\Omega)$, up to a subsequence. 
Let us show that, in fact, we can pass to a further subsequence so that $u_n\to u$ in $C^1(\overline\Omega)$. 
Denote the right-hand side of \eqref{eq:Dn} as $g_n$, that is,
$$
g_n(x) = \widetilde{\lambda}_n |u_n(x)|^{p-2}u_n(x)
+
f_n(x), \quad x \in \Omega. 
$$
Observe that the sequence $\{g_n\}$ is bounded in $L^\gamma(\Omega)$. 
Indeed, since $\{\|f_n\|_\gamma\}$ and $\{\widetilde{\lambda}_n\}$ are bounded, it is sufficient to justify the boundedness of $\{\|u_n\|_\infty\}$. 
If $N < p$, then the result is a consequence of the boundedness of $\{\|\nabla u_n\|_p\}$ and the Morrey lemma. 
If $N \geq p$, then we use \cite[Proposition~A.1]{BobkovTanakaAM} to get
$$
\|u_n\|_\infty \leq C(1+ \|u_n\|_r) \leq C+ C\|\nabla u_n\|_p,
$$
where $r \in (p\gamma',p^*)$ and $C>0$ do not depend on $n$. 
In view of the boundedness of $\{\|\nabla u_n\|_p\}$, we conclude that $\{\|u_n\|_\infty\}$ is also bounded, and hence so is $\{\|g_n\|_\gamma\}$. 
Therefore, \cite[Proposition~A.3]{BobkovTanakaAM} yields the existence of $\beta \in (0,1)$ and $C>0$ such that $u_n \in C^{1,\beta}(\overline{\Omega})$ and $\|u_n\|_{C^{1,\beta}(\overline{\Omega})} \leq C$ for any $n$. 
By the Arzel\`a-Ascoli theorem, there exists $v$ such that $u_n \to v$ in $C^1(\overline\Omega)$, up to a subsequence. It is not hard to see that $v \equiv u$ in $\Omega$.
Indeed, we have
$$
\|u_n - v\|_{W^{1,p}(\Omega)}
\leq 
|\Omega|^{1/p} \left(
\|u_n - v\|_{\infty}
+
\|\nabla u_n - \nabla v\|_{\infty}
\right) \to 0,
$$
and hence $v$ is a $W^{1,p}$-limit of $\{u_n\}$. 
Recalling that $u$ is a weak $W_0^{1,p}$-limit of $\{u_n\}$, we conclude that $u$ coincides with $v$ a.e.\ in $\Omega$. 
Passing to the limit in the weak formulation of \eqref{eq:Dn}, 
we further deduce that $u$ weakly solves \eqref{eq:D0}. 
Noting that $f$ is nontrivial by the assumption \eqref{eq:F}, we see that $u$ is also nontrivial. 
However, if $\lambda^*=\lambda_1$, then the problem \eqref{eq:D0} cannot have solutions by \cite[Corollary~2.19]{BobkovTanakaAM}. Thus, $\lambda^* > \lambda_1$. 

Assume now that $\lambda^* \in (\lambda_1, \lambda_2)$, and let us show that $\{\|u_n\|_p\}$ is bounded.
Suppose, by contradiction, that there exists a subsequence of $\{u_n\}$ along which $\|u_n\|_p\to +\infty$.  
Let us normalize $u_n$ by considering $w_n:=u_n/\|u_n\|_p$, so that $\|w_n\|_p=1$ and $w_n$ is a weak solution of 
\begin{equation}\label{eq:Dn2}
-\Delta_p w_n
=
\widetilde{\lambda}_n |w_n|^{p-2}w_n
+
\frac{f_n(x)}{\|u_n\|_p^{p-1}}
\quad \text{in}~ \Omega, \quad w_n=0 \quad\text{on}~\partial\Omega.
\end{equation}
Taking $w_n$ as a test function in the weak formulation of \eqref{eq:Dn2}, we get
\begin{equation}\label{eq:proof:amp1}
\|\nabla w_n\|_p^p
=
\widetilde{\lambda}_n
+
\frac{\int_\Omega f_n w_n\,dx}{\|u_n\|_p^{p-1}}.
\end{equation}
Arguing as in the derivation of \eqref{eq:proof:amp0}, we use the H\"older inequality 
to get
\begin{equation}\label{eq:proof:amp2}
\|\nabla w_n\|_p^p
\leq 
C
+
C\frac{\|\nabla w_n\|_p}{\|u_n\|_p^{p-1}},
\end{equation}
where $C>0$ does not depend on $n$. 
Since $\{\|u_n\|_p\}$ diverges, we deduce that $\{w_n\}$ is bounded in $W_0^{1,p}(\Omega)$, and hence $w_n \to w$ weakly in $W_0^{1,p}(\Omega)$ and strongly in $L^{\gamma'}(\Omega)$ for some $w \in W_0^{1,p}(\Omega)$, up to a subsequence. 
Arguing as above, one can also prove that $w_n \to w$ in $C^1(\overline{\Omega})$. 
Thus, we have $\|w\|_p=1$ and, passing to the limit in the weak formulation of \eqref{eq:Dn2}, we conclude that $\lambda^*$ is an eigenvalue and $w$ is a corresponding eigenfunction. 
However, recall that there are no eigenvalues between $\lambda_1$ and $\lambda_2$, see, e.g., \cite{anane1987}.  
This contradiction shows that the whole sequence $\{u_n\}$ is bounded in $L^p(\Omega)$. 
\end{proof}

\begin{theorem}\label{thm:AMP}
Let $\{f_n\}$ and $f$ satisfy \eqref{eq:F}, and $f_n\to f$ weakly in $L^\gamma(\Omega)$. 
Let $\{\widetilde{\lambda}_n\} \subset (\lambda_1,\lambda_2)$ be such that $\widetilde\lambda_n \to \lambda_1$. 
Let $u_n$ be a solution of \eqref{eq:Dn}. 
Then
\begin{equation}\label{eq:Un}
 (\widetilde{\lambda}_n-\lambda_1)^{1/(p-1)}u_n
 \to 
 -\left(\int_\Omega f\varphi_1\,dx\right)^{1/(p-1)}\varphi_1
 \quad\text{in }C^1(\overline\Omega).
\end{equation}
In particular, $u_n < 0$ in $\Omega$ and $\partial_\nu u_n >0$ on $\partial\Omega$ for any sufficiently large $n$.
\end{theorem}
\begin{proof}
Applying Proposition~\ref{prop:AMP-convergence} to any subsequence of $\{u_n\}$, we see that $\|u_n\|_p\to +\infty$. 
Let us normalize $u_n$ by considering $w_n:=u_n/\|u_n\|_p$, so that $\|w_n\|_p=1$. 
Arguing as in the proof of Proposition~\ref{prop:AMP-convergence}, we deduce that \eqref{eq:Dn2}, \eqref{eq:proof:amp1}, and \eqref{eq:proof:amp2} are also satisfied. In particular, $\{w_n\}$ is bounded in $W_0^{1,p}(\Omega)$. 
Therefore, \eqref{eq:proof:amp1} and $\widetilde\lambda_n \to \lambda_1$ yield
$$
\|\nabla w_n\|_p^p\to\lambda_1.
$$
Recalling that $\|w_n\|_p=1$, the variational characterization \eqref{eq:lambda1} of $\lambda_1$ and its simplicity imply that $w_n\to \sigma\varphi_1$ in $W^{1,p}_0(\Omega)$ and hence in $L^{\gamma'}(\Omega)$, up to a subsequence,
where $\sigma\in\{-1,1\}$.
In fact, we have $w_n\to \sigma\varphi_1$ in $C^1(\overline\Omega)$, up to a further subsequence.
Indeed, denoting the right-hand side of \eqref{eq:Dn2} as $g_n$, 
it can be proved as in Proposition~\ref{prop:AMP-convergence} that $\{g_n\}$ is bounded in $L^\gamma(\Omega)$, which yields the desired convergence via the results of \cite{BobkovTanakaAM} and the Arzel\`a-Ascoli theorem.

If $\sigma=1$, then $w_n>0$ in $\Omega$ and hence $u_n>0$ in $\Omega$ for all sufficiently large $n$.
However, this contradicts \cite[Proposition~2.17]{BobkovTanakaAM} (or, equivalently, the combination of \cite[Theorem~2.1]{Alleg} and \cite{vaz}), which states that \eqref{D} has no nonnegative solutions for any $\lambda>\lambda_1$. 
Therefore, we have $\sigma=-1$, i.e., $w_n\to -\varphi_1$ in $C^1(\overline\Omega)$.
This yields $w_n<0$ in $\Omega$ and $\partial_\nu w_n>0$ on $\partial\Omega$ for all sufficiently large $n$, and the same inequalities are satisfied for $u_n$. 

In order to obtain \eqref{eq:Un}, it remains to identify the asymptotic behavior of $\|u_n\|_p$.
Observe that $\|w_n\|_p=1$ implies $\|\nabla w_n\|_p^p \geq \lambda_1$ in view of \eqref{eq:lambda1}. 
Hence, we deduce from \eqref{eq:proof:amp1}, the strong convergence $w_n \to -\varphi_1$ in $L^{\gamma'}(\Omega)$, and the weak convergence $f_n \to f$ in $L^\gamma(\Omega)$ that
\begin{equation}\label{eq:proof:amp3}
(\lambda_1 -
\widetilde{\lambda}_n)\|u_n\|_p^{p-1}
\leq 
(\|\nabla w_n\|_p^p -
\widetilde{\lambda}_n) \|u_n\|_p^{p-1}
=
\int_\Omega f_n w_n\,dx \to -\int_\Omega f \varphi_1 \,dx.
\end{equation}
To obtain the reverse inequality, we first note that \eqref{eq:bpl} and the $C^1(\overline{\Omega})$-convergence yield 
\begin{equation}\label{eq:asympt2}
\varphi_1(x)\asymp \text{dist}(x,\partial\Omega)
\quad \text{and} \quad 
-w_n(x) \asymp \text{dist}(x,\partial\Omega)
\end{equation}
in a neighborhood of $\partial\Omega$, and hence it is not hard to see that $\varphi_1^p/(-u_n)^{p-1} \in W_0^{1,p}(\Omega) \cap C^1(\Omega)$.  
Now we apply the Picone inequality \cite{Alleg} in the weak formulation of \eqref{eq:Dn} as follows:
\begin{align}
 \lambda_1
 =
 \int_\Omega |\nabla\varphi_1|^p\,dx
 &\ge
 \int_\Omega |\nabla (-u_n)|^{p-2} \nabla (-u_n) \cdot \nabla \left(\frac{\varphi_1^p}{(-u_n)^{p-1}}\right)\,dx
 \\
 &=
 \int_\Omega (\widetilde{\lambda}_n (-u_n)^{p-1}-f_n)\frac{\varphi_1^p}{(-u_n)^{p-1}}\,dx
 =
 \widetilde{\lambda}_n-\int_\Omega f_n\frac{\varphi_1^p}{(-u_n)^{p-1}}\,dx.
\end{align}
Multiplying by $\|u_n\|_p^{p-1}$, we obtain
\begin{equation}\label{eq:proof:amp4}
 (\widetilde{\lambda}_n-\lambda_1)\|u_n\|_p^{p-1}
 \le
 \int_\Omega f_n\frac{\varphi_1^p}{(-w_n)^{p-1}}\,dx.
\end{equation}
Since $w_n \to -\varphi_1$ in $C^1(\overline{\Omega})$ and $\varphi_1$ satisfies \eqref{eq:bpl}, we recall \eqref{eq:asympt2} and get 
$$
\frac{\varphi_1^p}{(-w_n)^{p-1}}
=
\left(\frac{\varphi_1}{-w_n}\right)^{p-1} \varphi_1 \to \varphi_1
\quad \text{in}~ \overline{\Omega}.
$$ 
In particular, this convergence is strong in $L^{\gamma'}(\Omega)$. Thus, by the weak convergence $f_n \to f$ in $L^\gamma(\Omega)$, we deduce that 
\begin{equation}\label{eq:proof:amp5}
 \int_\Omega f_n\frac{\varphi_1^p}{(-w_n)^{p-1}}\,dx
 \to
 \int_\Omega f \varphi_1\,dx.
\end{equation}
Combining \eqref{eq:proof:amp3}, \eqref{eq:proof:amp4}, \eqref{eq:proof:amp5}, we arrive at
$$
 (\widetilde{\lambda}_n-\lambda_1) \|u_n\|_p^{p-1}
 \to
 \int_\Omega f\varphi_1\,dx.
$$
Recalling that $w_n = u_n/\|u_n\|_p\to-\varphi_1$ in $C^1(\overline\Omega)$, we derive the asymptotic formula \eqref{eq:Un} along an appropriately chosen subsequence of $\{u_n\}$. 
Applying the same analysis to any subsequence, we conclude that the whole sequence $\{u_n\}$ satisfies \eqref{eq:Un}.
\end{proof}

\begin{lemma}\label{lem:properties-of-U}
The set $\mathcal{U}_\eta \cup \{0\}$ is a convex cone. 
Moreover, if $\{f_n\} \subset \mathcal{U}_\eta$ is such that $\|f_n\|_\gamma=1$, 
then any weak $L^\gamma(\Omega)$-limit of $\{f_n\}$ belongs to $\mathcal{U}_\eta$.
\end{lemma}
\begin{proof}
If $f,g \in \mathcal{U}_\eta$, then, for any $s,t\ge0$,
$$
\int_\Omega (sf+tg)\varphi_1\,dx
\ge
\eta\bigl(s\|f\|_\gamma+t\|g\|_\gamma\bigr)
\ge
\eta\|sf+tg\|_\gamma,
$$
so that $sf+tg \in \mathcal{U}_\eta \cup \{0\}$. 
That is, $\mathcal{U}_\eta \cup \{0\}$ is a convex cone. 

Let $\{f_n\} \subset \mathcal{U}_\eta$ be such that $\|f_n\|_\gamma=1$. 
Thus, $f_n \to f$ weakly in $L^\gamma(\Omega)$ for some $f \in L^\gamma(\Omega)$, up to a subsequence.  
The weak convergence implies that $f \geq 0$ in $\Omega$ and 
$$
\int_\Omega f\varphi_1\,dx
=
\lim_{n\to+\infty}\int_\Omega f_n\varphi_1\,dx
\geq
\eta \liminf_{n\to+\infty} \|f_n\|_\gamma
= 
\eta.
$$
Hence, $f\not\equiv0$ in $\Omega$, and the weak lower semicontinuity $\|f\|_\gamma \leq \liminf_{n\to+\infty} \|f_n\|_\gamma$ yields 
$$
\Theta(f)
= 
\frac{\int_\Omega f\varphi_1\,dx}{\|f\|_\gamma} 
\geq 
\limsup_{n\to+\infty} 
\frac{\int_\Omega f_n \varphi_1\,dx}{\|f_n\|_\gamma} \geq \eta,
$$
so that $f \in \mathcal{U}_\eta$.
\end{proof}

\begin{lemma}\label{lem:implication}
Let $p=2$ and $N=1$.  
Then there exist $\delta>0$ and $\{f_n\}$ satisfying \eqref{eq:F} such that $\Theta(f_n) \to 0$, but $\lambda_{f_n} \geq \lambda_1 + \delta$ for any $n$. 
\end{lemma}
\begin{proof}
Let $\Omega=(0,1)$, so that $\lambda_k=(k\pi)^2$, $\varphi_1 = \sqrt{2}\sin(\pi x)$, and $\|\varphi_1\|_\infty=\sqrt2$.  
Choose a sequence $\varepsilon_n \searrow 0$ and define
$$
I_n=\left(\frac13-\varepsilon_n,\frac13+\varepsilon_n\right) \subset \Omega,
\quad 
f_n=\chi(I_n)
\quad \text{in}~ \Omega,
$$
where $\chi(I_n)$ is the characteristic function of $I_n$. 
Noting that $\|f_n\|_\gamma=(2\varepsilon_n)^{1/\gamma}$, we have
$$
 0<\Theta(f_n)
 \leq\sqrt2\,(2\varepsilon_n)^{1-1/\gamma} \to 0.
$$

For $\lambda \in(\pi^2,4\pi^2)$, the Dirichlet Green's function corresponding to $-\partial_x^2-\lambda$ is
$$
 G_\lambda(x,y)=
 \frac{\sin(\sqrt{\lambda}\min\{x,y\})\,
       \sin(\sqrt{\lambda}(1-\max\{x,y\}))}{\sqrt{\lambda}\sin \sqrt{\lambda}},
$$
see, e.g., \cite[Example~3.3.3]{duffy}. 
Let us find the sign of $G_\lambda$. 
Since $\sqrt{\lambda} \in (\pi,2\pi)$, we have $\sin \sqrt{\lambda}<0$. 
Assume now that $\sqrt{\lambda} \in (\pi, \pi/(2/3+\varepsilon_n))$. 
Then for every $y\in I_n$ both $\sqrt{\lambda}y$ and $\sqrt{\lambda}(1-y)$ belong to $(0,\pi)$.
Since $\min\{x,y\}\leq y$ and $1-\max\{x,y\}\leq1-y$, it follows that $G_\lambda(x,y)<0$ for all $x\in(0,1)$ and $y\in I_n$.
Therefore, the unique solution $u$ of \eqref{D} satisfies $u(x)=\int_{I_n}G_\lambda(x,y)\,dy<0$ for $x \in \Omega$.   
By differentiating $G_\lambda$ at the endpoints, we get
$$
 u'(0)=\int_{I_n}\frac{\sin(\sqrt{\lambda}(1-y))}{\sin \sqrt{\lambda}}\,dy<0,\qquad
 u'(1)=-\int_{I_n}\frac{\sin(\sqrt{\lambda}y)}{\sin \sqrt{\lambda}}\,dy>0.
$$
That is, $u$ satisfies \ref{AMP} whenever $\sqrt{\lambda} \in (\pi, \pi/(2/3+\varepsilon_n))$. 
Consequently, $\lambda_{f_n}\geq\pi^2/(2/3+\varepsilon_n)^2 \geq \lambda_1 + \delta$ for some $\delta>0$ which does not depend on $n$. 
\end{proof}

\section{Proofs}\label{sec:proofs}

\begin{proof}[Proof of Theorem~\ref{thm:UAMP}]
Let $\eta>0$. 
Assume, by contradiction to the claim of the theorem, that there exists a sequence $\{f_n\} \subset \mathcal{U}_\eta$ such that $\lambda_{f_n} \searrow \lambda_1$. 
Since $\lambda_{c f_n}=\lambda_{f_n}$ for any $c > 0$ (see \eqref{eq:lambdacf}), we can additionally assume that $\|f_n\|_\gamma=1$. 
Consequently, $f_n \to f$ weakly in $L^\gamma(\Omega)$ for some $f \in L^\gamma(\Omega)$, up to a subsequence. 
Lemma~\ref{lem:properties-of-U} further yields $f \in \mathcal{U}_\eta$. 
Recall that the problem \eqref{D} has a solution for any $\lambda \in (\lambda_1,\lambda_2)$.
Thus, by the definition of $\lambda_{f_n}$, there exists a sequence $\widetilde{\lambda}_n \searrow \lambda_1$ such that each $\widetilde{\lambda}_n \geq \lambda_{f_n}$ and a corresponding solution $u_n$ of \eqref{eq:Dn} violates \ref{AMP}. 
That is, $u_n$ does not satisfy $u_n<0$ in $\Omega$ or $\partial_\nu u_n >0$ on $\partial \Omega$. 
In either case, we get a contradiction to Theorem~\ref{thm:AMP}.
\end{proof}

\begin{proof}[Proof of Theorem~\ref{thm:LSC}]
Since both $\{f_n\}$ and $f$ satisfy \eqref{eq:F}, and $f_n \to f$ weakly in $L^\gamma(\Omega)$, there exists $\eta>0$ such that $\{f_n\} \subset \mathcal{U}_\eta$ and $f \in \mathcal{U}_\eta$. 
Indeed, it is sufficient to observe that 
$$
\Theta(f_n) 
= 
\frac{\int_\Omega f_n \varphi_1\,dx}{\|f_n\|_\gamma} 
\geq
\frac{\inf_n \int_\Omega f_n \varphi_1\,dx}{\sup_n \|f_n\|_\gamma} = \text{const}>0. 
$$
Thus, by Theorem~\ref{thm:UAMP}, there exists $\varepsilon>0$ such that $\lambda_f, \lambda_{f_n} \geq \lambda_1+\varepsilon$ for any $n$. 
If $(\lambda_1+\varepsilon,\min\{\lambda_f,\lambda_2\}) = \emptyset$, then the proof is complete. 
Otherwise, we fix any $a \in (\lambda_1+\varepsilon,\min\{\lambda_f,\lambda_2\})$. 
Let us show that
\begin{equation}\label{eq:proof:lws1}
 a\le \liminf_{n\to+\infty}\lambda_{f_n}.
\end{equation}
Suppose, by contradiction, that \eqref{eq:proof:lws1} does not hold.
That is, passing to a subsequence of indices, there exists a sequence $\{\widetilde{\lambda}_n\} \subset [\lambda_1+\varepsilon,a]$ and a sequence $\{u_n\}$ of corresponding solutions of the problem \eqref{eq:Dn} such that each $u_n$ does not satisfy $u_n<0$ in $\Omega$ or $\partial_\nu u_n >0$ on $\partial \Omega$. 
Passing to a further subsequence, we may assume that $\widetilde{\lambda}_n\to\lambda^*\in[\lambda_1+\varepsilon,a] \subset (\lambda_1,\lambda_2)$.

By Proposition~\ref{prop:AMP-convergence}, $u_n \to u$ in $C^1(\overline\Om)$, where $u$ is a solution of \eqref{eq:D0}. 
Since $\lambda^* \le a < \lambda_f$, \ref{AMP} says that $u<0$ in $\Omega$ and $\partial_\nu u > 0$ on $\partial \Omega$, and hence the same inequalities must be satisfied for $u_n$ with any sufficiently large $n$. 
This is a contradiction to our initial assumption on $u_n$, and hence \eqref{eq:proof:lws1} holds. 
Since $a \in (\lambda_1+\varepsilon,\min\{\lambda_f,\lambda_2\})$ was arbitrary, we send $a \to \min\{\lambda_f,\lambda_2\}$ and conclude the desired result. 
\end{proof}

\begin{proof}[Proof of Proposition~\ref{prop:upperbound} and Corollary~\ref{cor:eamp}]
	Let us take any $\lambda>\lambda_1$ such that \eqref{D} has a solution $u < 0$ in $\Omega$. 
	Since $\Omega_f^0 \neq \emptyset$, we fix any nonnegative function $\varphi \in C_0^{\infty}(\Omega_f^0)$ and extend it by zero outside of $\Omega_f^0$. 
	Recalling that $u \in C^1(\Omega)$, we have $u < 0$ in $\text{supp}\,\varphi$, and hence 
	$\varphi^p/(-u)^{p-1} \in \W(\Omega) \cap C^1(\Omega)$. 
	Applying the Picone inequality \cite{Alleg} in the weak formulation of \eqref{D}, we obtain
	\begin{align}
		-\intO |\nabla \varphi|^{p} \,dx
		&\leq
		\intO |\nabla u|^{p-2} \nabla u \cdot \nabla \left(\frac{\varphi^p}{(-u)^{p-1}}\right) dx
		\\
		\label{eq:prop:weak1}
		&=\lambda
		\intO \frac{|u|^{p-2}u}{(-u)^{p-1}}\,
		\varphi^p\, dx
		+
		\intO f
		\frac{\varphi^{p}}{(-u)^{p-1}} \,dx
		=
		-\lambda \intO \varphi^p\, dx.
	\end{align}
	Therefore,
	$$
	\lambda \int_{\Omega_f^0} \varphi^p \,dx \leq \int_{\Omega_f^0} |\nabla \varphi|^{p} \,dx.
	$$
	Since the nonnegative function $\varphi \in C_0^\infty(\Omega_f^0)$ was arbitrary, we derive $\lambda \leq \lambda_1(\Omega_f^0)$ from \eqref{eq:lambda1} and \cite[Lemma~1.23]{HKM}.
\end{proof}

\begin{proof}[Proof of Proposition~\ref{prop:convexity}]
Fix any $a\in(\lambda_1,\lambda_2)$ and $\lambda\in(\lambda_1,a)$. 
Denote by $u_{f}$ the (unique) solution of \eqref{D}. 
Since $p=2$, we have
$$
u_{sf+tg}
=
s u_{f}
+
t u_{g}  \quad \text{for any}~ s,t \in \mathbb{R}.
$$
If $f,g\in\mathcal A_a$, then $u_{f}, u_{g}<0$ in $\Omega$ and $\partial_\nu u_f,\partial_\nu u_g > 0$ on $\partial\Omega$, and hence $u_{sf+tg}$ satisfies the same inequalities, provided $s,t\ge0$, $s+t>0$. 
Since this holds for every $\lambda\in(\lambda_1,a)$, we get $sf+tg\in\mathcal A_a$. 
Noting that, in the linear case, $\lambda_f, \lambda_g \leq \lambda_2$ (see, e.g., \cite[Theorem~1.1]{BDI}), we send $a$ to $\min\{\lambda_f,\lambda_g\}$ and arrive at $\min\{\lambda_f,\lambda_g\} \leq \lambda_{sf+tg}$. 
\end{proof}

\begin{proof}[Proof of Proposition~\ref{prop:eamp-linear}]
Let $m \in \mathbb{N}$. 
Since $\varphi_i \in C^{1,\beta}(\overline{\Omega})$ for any $i \in \mathbb{N}$, and $\varphi_1$ satisfies \eqref{eq:bpl}, there exists $M>0$ such that 
\begin{equation}\label{eq:eamp:0}
|\varphi_i| \leq M \varphi_1 \quad \text{in}~ \Omega \quad \text{for each}~ i=1,\ldots,m.
\end{equation}
Let $\{\lambda_k\}$ be the spectrum of the Dirichlet Laplacian in $\Omega$. 
Take any $\lambda>\lambda_m$ and $c_1,\ldots,c_m\in\mathbb{R}$, and consider the function
$$
  u_\lambda:=-\sum_{i=1}^m\frac{c_i}{\lambda-\lambda_i}\varphi_i.
$$
We see that $u_\lambda \in W_0^{1,2}(\Omega) \cap C^{1,\beta}(\overline{\Omega})$. 
Since $-\Delta\varphi_i=\lambda_i\varphi_i$, we obtain 
\begin{equation}\label{eq:eamp:1}
  -\Delta u_\lambda - \lambda u_\lambda
  = \sum_{i=1}^m c_i\varphi_i
  =:
  f \quad \text{in}~ \Omega. 
\end{equation}
That is, $u_\lambda$ is a solution of \eqref{D}. 
Assume that $\lambda\geq 2\lambda_m$. 
In view of \eqref{eq:eamp:0} and \eqref{eq:eamp:1}, we have 
$$
  |\lambda u_\lambda+f|
  =
  |\Delta u_\lambda| 
  =
  \left|\sum_{i=1}^m\frac{c_i\lambda_i}{\lambda-\lambda_i}\varphi_i\right|
  \leq 
  2M\frac{\lambda_m}{\lambda} \sum_{i=1}^m |c_i| \,\varphi_1 \quad \text{in}~ \Omega.
$$
Assume now that $c_1,\ldots,c_m\in\mathbb{R}$ are chosen so that $f\geq \kappa\varphi_1$ in $\Omega$ for some $\kappa>0$. 
Then we obtain 
$$
  \lambda u_\lambda \leq \left(-\kappa + 2M\frac{\lambda_m}{\lambda} \sum_{i=1}^m |c_i|\right) \varphi_1  \quad \text{in}~ \Omega.
$$
Hence, for any sufficiently large $\lambda$, we get $u_\lambda<0$ in $\Omega$, which yields $\widetilde\lambda_f=+\infty$.
\end{proof}

\begin{proof}[Proof of Theorem~\ref{thm:anticomparison}]
Unlike the argument in \cite{CuestaTakac-strongcomparison1}, our approach is based on the asymptotic profile of negative solutions as $\lambda \searrow \lambda_1$ obtained in Theorem~\ref{thm:AMP}.
Since $0 \leq f \leq g$, $f \not\equiv 0$, $f \not\equiv g$, and $\varphi_1>0$ in $\Omega$, we have
\begin{equation}\label{eq:proof:ACP1}
0 < \intO f \varphi_1 \,dx < \intO g \varphi_1 \,dx.
\end{equation}
Suppose, contrary to the claim of the theorem, that there exists a decreasing sequence $\widetilde{\lambda}_n  \to \lambda_1$ and solutions $u_{n}, v_n$ of \eqref{eq:D1}, \eqref{eq:D2} (with $\lambda=\widetilde{\lambda}_n$), respectively, such that \eqref{eq:thm:ACP} does not hold. 
Observe that $u_n,v_n < 0$ in $\Omega$ and $\partial_\nu u_n$, $\partial_\nu v_n > 0$ on $\partial \Omega$ for any sufficiently large $n$, in view of \ref{AMP}. 
Applying Theorem~\ref{thm:AMP} to $u_n$ and $v_n$, and then subtracting the corresponding expressions \eqref{eq:Un}, we arrive at
\begin{equation}
 (\widetilde{\lambda}_n-\lambda_1)^{1/(p-1)} (v_n-u_n)
 \to 
 -\left[
 \left(\int_\Omega g\varphi_1\,dx\right)^{1/(p-1)}
-
\left(\int_\Omega f\varphi_1\,dx\right)^{1/(p-1)}
 \right]
 \varphi_1
 \quad\text{in }C^1(\overline\Omega).
\end{equation}
Therefore, by \eqref{eq:proof:ACP1} and \ref{AMP}, we get $v_n<u_n < 0$ in $\Omega$ and $\partial_\nu v_n > \partial_\nu u_n > 0$ on $\partial\Omega$ for any sufficiently large $n$. 
This contradiction completes the proof. 
\end{proof}

\begin{proof}[Proof of Proposition~\ref{prop:anticomparison-linear}] 
Denote $\Lambda_{f,g} = \min\{\lambda_f,\lambda_{g-f}\} \leq \lambda_2$. 
Let us take any $\lambda \in (\lambda_1, \Lambda_{f,g})$ and let $u,v$ be the solutions of \eqref{eq:D1}, \eqref{eq:D2}, respectively. 
Since $\lambda<\lambda_{f}$, \ref{AMP} applies to $u$. 
By the linearity, the function $w=v-u$ solves
$$
 -\Delta w=\lambda w+(g-f) \quad \text{in}~ \Omega, \quad w=0 \quad \text{on }\partial\Omega.
$$
Since $f \leq g$, $f\not\equiv g$ in $\Omega$, and $\lambda<\lambda_{g-f}$, \ref{AMP} gives $w<0$ in $\Omega$ and $\partial_\nu w > 0$ on $\partial\Omega$. 
That is, the anticomparison principle \eqref{eq:thm:ACP} holds for any $\lambda \in (\lambda_1, \Lambda_{f,g})$. 

Let $\lambda=\Lambda_{f,g}$ and assume that $\Lambda_{f,g}=\lambda_f$. 
Then there exists a solution of \eqref{D} which violates \ref{AMP}. Indeed, if $\lambda_f=\lambda_2$, then \cite[Proposition~1.2 (i)]{BDI} implies that $\int_\Omega f \varphi_2\,dx=0$ for any second eigenfunction $\varphi_2$. 
The Fredholm alternative yields the existence of a continuum of solutions of \eqref{D} of the form $u + c\varphi_2$. Since any $\varphi_2$ is sign-changing, so is $u + c\varphi_2$ for an appropriate $c \in \mathbb{R}$. 
If $\lambda_f < \lambda_2$ and we suppose that the solution $u$ of \eqref{D} (with $\lambda=\lambda_f$) satisfies \ref{AMP}, then the continuity result of Proposition~\ref{prop:AMP-convergence} implies that the solutions of \eqref{D} also satisfy \ref{AMP} for any $\lambda$ in a sufficiently small neighborhood of $\lambda_f$, which contradicts the maximality of $\lambda_f$, see \eqref{eq:lambdaf}. 
The same arguments cover the case $\lambda=\Lambda_{f,g}=\lambda_{g-f}$. 
Therefore, we conclude that if $\lambda=\Lambda_{f,g}$, then \eqref{eq:thm:ACP} does not hold, i.e., $\Lambda_{f,g}$ is the endpoint of validity of the anticomparison principle.
\end{proof}

\begin{proof}[Proof of Theorem~\ref{thm:AMP-morrey}]
Recall that we assume $p=2$. 
We start by discussing a few auxiliary facts. 
Since $f\in L^{q,\mu}(\Omega)$, we have $f \in L^q(\Omega)$ by \eqref{eq:propertiesMorrey}, and hence the problem
\begin{equation}\label{eq:Dmorrey}
-\Delta v = f(x) ~\text{in}~ \Omega, \quad v=0 ~\text{on}~\partial\Omega,
\end{equation}
has a unique strong solution $v \in W^{2,q}(\Omega)\cap W^{1,q}_0(\Omega)$, and $v$ satisfies
\begin{equation}\label{eq:laplacecontind}
\|v\|_{W^{2,q}(\Omega)} \leq C \|f\|_q,
\end{equation}
where $C>0$ does not depend on $f$, see \cite[Theorem~9.15 and Lemma~9.17]{GT}. 
This implies that the operator $-\Delta:W^{2,q}(\Omega)\cap W^{1,q}_0(\Omega)\to L^q(\Omega)$ is an isomorphism, and hence the inverse operator is continuous by \eqref{eq:laplacecontind}. Moreover, thanks to the compactness of the embedding $W^{2,q}(\Omega) \hookrightarrow L^q(\Omega)$, the inverse operator $(-\Delta)^{-1}: L^q(\Omega) \to L^q(\Omega)$ is compact. 
Consequently, by the Fredholm alternative, the problem \eqref{D} has a unique strong solution $u$, provided $\lambda$ is not an eigenvalue. 
In particular, such $u$ exists and is unique for any $\lambda\in(\lambda_1,\lambda_2)$. 
In what follows, by referring to the strong solution of \eqref{D} or \eqref{eq:Dmorrey}, we always mean that it belongs to $W^{2,q}(\Omega)\cap W^{1,q}_0(\Omega)$. 

By \cite[Corollary~4.1]{DPR1999} and the assumptions \eqref{eq:morrey:assumptions}, the strong solution of \eqref{eq:Dmorrey} actually belongs to $C^{1,\alpha}(\overline{\Omega})$, where $\alpha=1-(N-\mu)/q \in (0,1)$. 
Although it is not stated explicitly in \cite{DPR1999}, we also have an a priori $C^{1,\alpha}(\overline{\Omega})$-estimate. 
Namely, by the proof of \cite[Theorem~3.4]{DPR1999}, the operator $-\Delta:W^{2,q,\mu}(\Omega)\cap W^{1,q}_0(\Omega)\to L^{q,\mu}(\Omega)$ is continuous, where $W^{2,q,\mu}(\Omega)$ is the Banach space of functions $u \in W^{2,q}(\Omega)$ such that $D^2u \in L^{q,\mu}(\Omega)$ and equipped with the norm
$$
\|u\|_{W^{2,q,\mu}(\Omega)} = \|u\|_{q} + \|D^2u\|_{q,\mu},
$$
see \cite[Section~2]{DPR1999}. 
Therefore, by \cite[Theorem~3.4 and Corollary~4.1]{DPR1999}, \textit{every} $u \in W^{2,q,\mu}(\Omega)\cap W^{1,q}_0(\Omega)$ belongs to $C^{1,\alpha}(\overline{\Omega})$. 
Consider now the embedding operator $J: W^{2,q,\mu}(\Omega)\cap W^{1,q}_0(\Omega) \to C^{1,\alpha}(\overline{\Omega})$ defined by $Ju=u$. 
If $u_n \to u$ in $W^{2,q,\mu}(\Omega)\cap W^{1,q}_0(\Omega)$ and $u_n \to w$ in $C^{1,\alpha}(\overline{\Omega})$, then both convergences imply the convergence in $L^q(\Omega)$, which yields $u=w$ a.e.\ in $\Omega$. 
In other words, the graph of $J$ is closed. 
Since both $W^{2,q,\mu}(\Omega)\cap W^{1,q}_0(\Omega)$ and $C^{1,\alpha}(\overline{\Omega})$ are Banach spaces, the closed graph theorem then states that $J$ is continuous, that is, there exists $C_1>0$ such that
\begin{equation}\label{eq:proof:morrey1}
\|u\|_{C^{1,\alpha}(\overline{\Omega})} \leq C_1 \|u\|_{W^{2,q,\mu}(\Omega)} 
\quad \text{for any}~ u \in W^{2,q,\mu}(\Omega)\cap W^{1,q}_0(\Omega).
\end{equation}
On the other hand, by the proof of \cite[Theorem~3.4]{DPR1999}, the inverse of the operator $-\Delta:W^{2,q,\mu}(\Omega)\cap W^{1,q}_0(\Omega)\to L^{q,\mu}(\Omega)$ is also continuous, and hence, in addition to \cite[Theorem~3.4, Eq.(3.8)]{DPR1999}, the strong solution $v$ of \eqref{eq:Dmorrey} satisfies
\begin{equation}\label{eq:proof:morrey2}
\|v\|_{W^{2,q,\mu}(\Omega)} \leq C_2 \|f\|_{q,\mu},
\end{equation}
where $C_2>0$ does not depend on $f$.  
Combining \eqref{eq:proof:morrey1} and \eqref{eq:proof:morrey2}, we conclude that 
\begin{equation}\label{eq:proof:morrey3}
\|v\|_{C^{1,\alpha}(\overline{\Omega})} \leq C_1 C_2 \|f\|_{q,\mu}.
\end{equation}

\medskip
Suppose now, by contradiction to \ref{AMP}, that there exists a decreasing sequence 
$\widetilde{\lambda}_n \to \lambda_1$ and the corresponding sequence $\{u_n\}$ of strong solutions
of \eqref{eq:Dn} such that each $u_n$ does not satisfy  $u_n < 0$ in $\Omega$ or $\partial_\nu u_n > 0$ on $\partial \Omega$. 
We have $ \|u_n\|_{\infty} \to +\infty$. 
Indeed, if $\{u_n\}$ is bounded in $L^\infty(\Omega)$, up to a subsequence, then the sequence of the right-hand sides $\{\widetilde{\lambda}_nu_n+f\}$ of \eqref{eq:Dn} is bounded in $L^{q,\mu}(\Omega)$, and hence $\{u_n\}$ is bounded in $C^{1,\alpha}(\overline{\Omega})$ by \eqref{eq:proof:morrey3}. 
Consequently, the Arzel\`a-Ascoli theorem yields $u_n \to u$ in $C^1(\overline{\Omega})$ for some $u$, up to a subsequence. 
Passing to the limit in the weak formulation of \eqref{eq:Dn}, we further conclude that $u$ is a \textit{weak} solution of \eqref{D} with $\lambda=\lambda_1$, i.e.,
\begin{equation}\label{eq:Dlambda1-morry}
 -\Delta u=\lambda_1u+f(x)\quad\text{in }\Omega,
 \quad u=0\quad\text{on }\partial\Omega.
\end{equation}
Testing the weak formulation of \eqref{eq:Dlambda1-morry} by $\varphi_1$, we derive
$$
 \int_\Omega f\varphi_1\,dx
 =\int_\Omega\nabla u\cdot\nabla\varphi_1\,dx
  -\lambda_1\int_\Omega u\varphi_1\,dx=0,
$$
which contradicts our assumption $\int_\Omega f\varphi_1\,dx>0$. 
Therefore, $ \|u_n\|_{\infty} \to +\infty$. 

Consider the normalized function $w_n = u_n/\|u_n\|_\infty$, so that $w_n$ is a strong solution of 
\begin{equation}\label{eq:Dnmorry2}
 -\Delta w_n=\widetilde{\lambda}_n w_n+\frac{f}{\|u_n\|_\infty}
 \quad\text{in }\Omega,
 \quad w_n=0\quad\text{on }\partial\Omega.
\end{equation}
The right-hand sides of \eqref{eq:Dnmorry2} are uniformly bounded in
$L^{q,\mu}(\Omega)$, and therefore, as above, $\{w_n\}$ is bounded in $C^{1,\alpha}(\overline{\Omega})$ and the Arzel\`a-Ascoli theorem implies $w_n \to w$ in $C^1(\overline{\Omega})$ for some $w$, up to a subsequence. 
Notice that $\|w\|_\infty=1$, so that $w$ is nontrivial. 
By passing to the limit in the weak formulation of \eqref{eq:Dnmorry2} and recalling that $\widetilde{\lambda}_n \to \lambda_1$, we conclude that $w$ is the first eigenfunction, that is, $w=c\varphi_1$ for some $c\ne0$.

Let us find the sign of $c$. Testing the weak formulation of \eqref{eq:Dnmorry2} by $\varphi_1$, we obtain
$$
 (\lambda_1-\widetilde{\lambda}_n)\int_\Omega w_n\varphi_1\,dx
 =
 \frac{1}{\|u_n\|_\infty}\int_\Omega f\varphi_1 \,dx
  >0.
$$
Since $\widetilde{\lambda}_n>\lambda_1$, it follows that $\int_\Omega w_n\varphi_1\,dx<0$, which yields $c<0$. 
Recalling that $\varphi_1$ satisfies \eqref{eq:bpl} and $w_n \to c\varphi_1$ in $C^1(\overline{\Omega})$, we conclude that 
$$
 w_n<0\quad\text{in }\Omega
 \quad \text{and} \quad 
  \partial_\nu w_n>0\quad\text{on }\partial\Omega
$$
for any sufficiently large $n$. 
Clearly, the same inequalities also hold for $u_n$, contradicting
the choice of the sequence. The proof is complete.
\end{proof}

\bigskip
\noindent
\textbf{Acknowledgments and historical remarks.} 
Section~\ref{sec:properties} mainly addresses questions raised by Evgeny~Yu.~Panov and Andrey~L.~Piatnitski during the O.A.~Ladyzhenskaya centennial conference on PDEs 2022 at the Euler International Mathematical Institute, St.~Petersburg. 
The idea of Section~\ref{sec:EAMP} emerged during the Seminar in Partial Differential Equations (SPDE) 2023 at the IIT Palakkad (India), and the author is thankful to Ashok Kumar for a discussion. 
The investigation of Section~\ref{sec:anticomparison} was proposed by the author back in 2014 when writing an application for a regional youth research grant. (That application was not supported.)  
Finally, the problem of obtaining \ref{AMP} in a scale of finer function spaces (see Section~\ref{sec:Morrey}) was explicitly raised in \cite[Section~1.1]{BobkovTanakaAM}. 
The author is thankful to Mieko Tanaka for related discussions. 
The author was supported in the framework of the development program of the Scientific Educational Mathematical Center of the Volga Federal District (agreement No.\ 075-02-2026-1332).

\phantomsection

\end{document}